\documentclass{jams-edited}

\usepackage{amssymb}

\usepackage{graphicx}

\usepackage[cmtip,all]{xy}

\usepackage{cite}
\usepackage{hyperref} 
\usepackage{color,soul} 
\usepackage{verbatim} 
 \usepackage{epsfig,epstopdf,verbatim,subfig,caption}
\usepackage[normalem]{ulem}  

\newtheorem{theorem}{Theorem}[section]
\newtheorem{lemma}[theorem]{Lemma}

\theoremstyle{definition}
\newtheorem{definition}[theorem]{Definition}

\theoremstyle{remark}
\newtheorem{remark}[theorem]{Remark}

\numberwithin{equation}{section}

\newcommand{\mc}[1]{\mathcal{#1}}

\newcommand{\rea}{\mc{R}e(\alpha)}

\newcommand{\ceila}[1]{\lceil \rea \rceil}

\begin{document}

\title[Existence and Uniqueness results for a class of Generalized ..... ]{Existence and Uniqueness results for a class of Generalized Fractional Differential Equations}


\author[U.N. Katugampola]{Udita N. Katugampola}
\address{Department of Mathematics, University of Delaware, Newark DE 19716, USA.}
\thanks{Department of Mathematics, University of Delaware, Newark DE 19716, USA. \\ \mbox{} \hspace{.2cm} Email: uditanalin@yahoo.com   \hfill Preprint submitted for publication.}
\email{uditanalin@yahoo.com} 

\keywords{Fractional Calculus, Existence, Uniqueness, Generalized fractional differential equations, \mbox{} \hspace{.2cm} Riemann-Liouville fractional derivative, Hadamard fractional derivative, Erd\'{e}lyi-Kober type derivative}

\date{mmm dd, 2014} 

\begin{abstract}
The author (Bull. Math. Anal. App. 6(4)(2014):1--15), introduced a new fractional derivative,
\[
{}^\rho \mathcal{D}_a^\alpha f (x) =  \frac{\rho^{\alpha-n+1 }}{\Gamma({n-\alpha})} \, \bigg(x^{1-\rho} \,\frac{d}{dx}\bigg)^n \int^x_a \frac{\tau^{\rho-1} f(\tau) }{(x^\rho - \tau^\rho)^{\alpha-n+1}}\, d\tau 
\]
which generalizes two familiar fractional derivatives, namely, the Riemann-Liouville and the Hadamard fractional derivatives to a single form. In this paper, we derive the existence and uniqueness results for a generalized fractional differential equation governed by the fractional derivative in question. 
\end{abstract}

\maketitle



\setcounter{section}{0}
\section{Introduction}\label{sec:1}

\setcounter{section}{1}
\setcounter{equation}{0}\setcounter{theorem}{0}

In recent years, Fractional Differential Equations (FDE) gain enormous attention among scientists due to the applications which were not possible with ordinary or partial differential equations of integer order. FDEs becomes a very successful tool in modeling anomalous diffusion and fractal-like nature \cite{B-T,kolwankar1,kolwankar5,kolwankar2,ano1,ano3}. Agrawal discusses diffusion and heat equations of fractional order in \cite{Agrawal-dif1,Agrawal-dif2,Agr-7}. Agrawal et al, Baleanu and others investigated the Boundary Value Problems for Fractional Differential Equations, for example see  \cite{AgBeHa-9,Agr-7}. Fractional dynamic models, fractional control systems, fractional population dynamics models and fractional fluid dynamics all involve at least one ordinary or partial fractional derivative.

In fractional calculus, fractional derivatives are defined via fractional integrals, see for example \cite{KilSri,Podlubny,Samko}. The author introduces an Erd\'{e}lyi-Kober type fractional integral operator in \cite{key-0} and uses that integral to define a new fractional derivative in \cite{key-00}, which generalizes the Riemann-Liouville and the Hadamard fractional derivatives to a single form and argued that it is not possible to derive Hadamard equivalence operators from the corresponding Erd\'{e}lyi-Kober type operators, thus making the new derivative more appropriate for modeling certain phenomena which undergo bifurcation-like behaviors. For further properties of the Erd\'{e}lyi-Kober operators, the interested reader is refereed to, for example, \cite{Kiryakova,Kir,Luchko}. According to the literature, the newly defined fractional operators are now known as Katugampola fractional integral and derivatives, respectively. It can be shown that the derivatives in question satisfy the fractional derivative criteria (test) given in \cite{what2,corr1}. These operators have applications in areas such as probability theory \cite{prob1}, theory of inequalities \cite{ineq1,HH1,HH2}, variational principle \cite{vara1}, numerical analysis \cite{numr1}, and Langevin equations \cite{lang1}. The interested reader is referred, for example, to \cite{Herr, u-1,u-2,u-3,u-4,u-5,u-6,u-7,u-8,u-9,u-10} for further results on these and similar operators. In \cite{key-000}, the author derived the Mellin transforms of the generalized fractional integrals and derivatives defined in \cite{key-0} and \cite{key-00}, respectively. In the same reference, the author also obtained a recurrence formula given by
\[
S(n,k)=\sum_{i=0}^r \big(m+(m-r)(n-2)+k-i-1\big)_{r-i}\binom{r}{i} S(n-1,k-i), \;\;  0< k\leq n,
\]
with $S(0,0)=1$ and $S(n,0)=S(n,k)=0$ for $n>0$ and $1+min\{r,m\}(n-1) < k $ or $k\leq 0$. Here $(\lambda)_r$ represent the Pochhammer's symbol given by $(\lambda)_r = \Gamma(\lambda+r)/\Gamma(\lambda), \; (\lambda)_0=1$. This generates a class of sequences for fixed $r, m \in \mathbb{N}$ which have similar characteristics as Stirling numbers of the $2^{nd}$ kind. These are also the coefficients generated by the differential operator $x^r\frac{d^m}{dx^m}$, for $r, m \in \mathbb{N}$.  Furthermore, it has shown that the case $r=1$ and $m=2$ also generates the same sequence as that of $r=1$ and $m=2$. It happens that the same result is true for any $r, m \in \mathbb{N}$ \cite{key-000}. 

The author also defined a new class of sequences generated by the differential operator $x^r\frac{d^m}{dx^m}$ for $r \in \{\frac{1}{2}, \frac{1}{3}, \frac{1}{4}, \frac{1}{5}, \ldots \}$, $m \in \mathbb{N}$ and in turn showed that the case $r=1/2$ and $m=1$ corresponds to the generalized Laguerre polynomials \cite{key-000}. Those sequences are listed on the Sloane's \textit{On-line Encyclopedia of Integer Sequences} (OEIS) at \url{http://oeis.org}. For example, see A223168-A223172 and A223523-A223532 \cite{key-000}.  

In this paper we study the existence and uniqueness results of a generalized fractional differential equation governed by the generalized fractional derivative in question. The paper is organized as follows: In the next section we give the necessary definitions and in the subsequent sections we develop the existence and uniqueness results. For simplicity we only consider the left-sided operators here. The right-sided operators can be treated similarly. 

\section{Generalized Fractional Integrals and Derivatives}

\begin{definition} The Riemann-Liouville fractional integrals and derivatives of order $\alpha \in \mathbb{C},\, Re(\alpha)\geq 0$ are given by
\begin{equation*}
(I^\alpha_{a+} f)(x) = \frac{1}{\Gamma(\alpha)}\int_a^x (x - \tau)^{\alpha -1} f(\tau) d\tau,  
\label{eq:lRLI}
\end{equation*}
and 
\begin{equation*}
(D^\alpha_{a+} f)(x)=\bigg(\frac{d}{dx}\bigg)^n \, \Big(I^{n-\alpha}_{a+} f\Big)(x), \quad x>a,
\label{eq:lRLD}
\end{equation*}
respectively, where $n = \lceil Re(\alpha) \rceil$ and $\Gamma(\alpha)$ is the Gamma function. 
\end{definition}

\begin{definition} The Hadamard fractional integral and derivative are given by,
\begin{equation*}
\mathbf{I}^\alpha_{a+} f(x) = \frac{1}{\Gamma(\alpha)}\int_a^x \Bigg(\log\frac{x}{\tau}\Bigg)^{\alpha -1} f(\tau)\frac{d\tau}{\tau}, 
\label{eq:HI}
\end{equation*}
and
\begin{equation*}
\mathbf{D}^\alpha_{a+}f(x)=\frac{1}{\Gamma(n-\alpha)}\Bigg(x\frac{d}{dx}\Bigg)^n\int_a^x \Bigg(\log\frac{x}{\tau}\Bigg)^{n-\alpha +1}f(\tau)\frac{d\tau}{\tau}, 
\label{eq:HD}
\end{equation*}
respectively, for $x > a\geq 0$, and $Re(\alpha) > 0$.
\end{definition}

Next we give the definitions of the generalized fractional operators introduced in \cite{key-0,key-00}:

\begin{definition}The generalized \emph{left-sided} fractional integral ${}^\rho I^\alpha_{a+}f$ of order $\alpha \in \mathbb{C} \; (Re(\alpha) > 0)$ is defined by
\begin{equation}
\big({}^\rho \mathcal{I}^\alpha_{a+}f\big)(x) = \frac{\rho^{1- \alpha }}{\Gamma({\alpha})} \int^x_a \frac{\tau^{\rho-1} f(\tau) }{(x^\rho - \tau^\rho)^{1-\alpha}}\, d\tau 
\label{eq:df1}
\end{equation}
for $x > a$, if the integral exists.

The generalized fractional derivative, corresponding to the generalized fractional integrals (\ref{eq:df1}), is defined, for $0 \leq a < x $, by
\begin{align}
\big({}^\rho &\mathcal{D}^\alpha_{a+}f\big)(x) = \bigg(x^{1-\rho} \,\frac{d}{dx}\bigg)^n\,\, \big({}^\rho \mathcal{I}^{n-\alpha}_{a+}f\big)(x)\nonumber\\
 &= \frac{\rho^{\alpha-n+1 }}{\Gamma({n-\alpha})} \, \bigg(x^{1-\rho} \,\frac{d}{dx}\bigg)^n \int^x_a \frac{\tau^{\rho-1} f(\tau) }{(x^\rho - \tau^\rho)^{\alpha-n+1}}\, d\tau
\label{eq:gd1}
\end{align}
if the integral exists.
\end{definition}

\begin{definition} The Caputo-type generalized fractional derivative, ${}^\rho_c D^\alpha_{a+}$ is defined via the above generalized fractional derivative (\ref{eq:gd1}) as follows 
\begin{equation}
{}^\rho_c D^\alpha_{a+}f(x) = \Bigg({}^\rho \mathcal{D}^\alpha_{a+}\,\bigg[f(t) - \sum^{n-1}_{k=0}\frac{f^{(k)}(a)}{k!}(t-a)^k\bigg]\Bigg)\,(x)
\label{eq:cgd}
\end{equation}
where $n =\lceil Re(\alpha) \rceil$.
\end{definition}
The interested reader may find Caputo-type modification of the Erd\'{e}lyi-Kober fractional derivative in the works of Luchko \cite{Luchko}. 

\section{Fractional Differential Equations}
In this section we prove the existence and uniqueness results of the solutions to Cauchy type problems for differential equations of fractional order on a finite interval of the real axis in the space of continuous functions. We only consider the fractional differential equations with Caputo type derivatives $\big({}^\rho_cD^\alpha_{a+}y\big)(x)$ due to their applications in several branches of sciences. The Riemann-Liouville type modifications will be discussed in a forthcoming paper.  

A differential equation with at least one fractional derivative is called a \emph{Fractional Differential Equation} (FDE). Thus, the nonlinear differential equation of fractional order $\alpha \,\, (\rea >0)$ on a finite interval $[a,\, b]$ of the real axis $\mathbb{R}$ has the form
\begin{equation}
\big({}^\rho_cD^\alpha_{a+}y\big)(x)=f[x,y(x)] \quad ;x > a.
\label{eq:ch4-eq1}
\end{equation}

\section{Existence results for Fractional Order ODEs}
As in the case of classical differential equations, the existence and uniqueness are essential for solving fractional differential equations. It is very important to know whether there is a solution to a given differential equation before hand. Otherwise, all the efforts to find a numerical or analytical solution will become worthless. With that in mind, here we shall discuss those results which are related to generalized fractional operators. We begin with the following Cauchy type initial value problem (IVP) of the form:
\begin{equation}
 \big({}^\rho_cD^\alpha_{0+}y\big)(x)=f(x,y(x)),\label{eq:pb-de1}
\end{equation}
with the initial condition
\begin{equation}
  D^k\,y(0)=y_0^{(k)}, \quad k = 0, 1, \ldots, m-1, \label{eq:pb-ic}
\end{equation}
where $m = \lceil \alpha \rceil$ and $\alpha \in \mathbb{R}$. The major result of this paper is a generalization of the findings in \cite{Diethelm-NFord,AFDE} for the existence that corresponds to the classical Peano existence theorem for first order differential equations.

\begin{theorem}[Existence]
Let $\alpha > 0$ and $m = \lceil \alpha \rceil$. Also let $y_0^{(0)}, \ldots, y_0^{(m-1} \in \mathbb{R}, \,\, K > 0 \,\, \text{and} \,\,h^* > 0.$ Define $G := \{(x, y) : x \in [0, h^*], \,\big|y - \sum_{k=0}^{m-1} \frac{x^k y_0^{(k)}}{k!} \big| \leq K \}$, and let the function $f:G \rightarrow \mathbb{R}$ be continuous. Further, define $ M := \sup_{(x, y) \in G} |f(x, z)|$ and
\begin{equation}
	h := \begin{cases}
		h^* & \text{if } M = 0,\\
                \min\Big\{h^*, \Big(\frac{K\Gamma(\alpha + 1)\rho^\alpha}{M}\Big)^{\frac{1}{\alpha}}\Big\} & \text{otherwise}.
\end{cases}
\label{eq:th1-con1}
\end{equation}
Then, there exists a function $y \in C[0, h]$ that solves the IVP \emph{(\ref{eq:pb-de1})} and \emph{(\ref{eq:pb-ic})}. 
\end{theorem}
The proof of this theorem uses the following lemma, which asserts the connection between a FDE and the corresponding Volterra type integral equation \cite{AFDE,Diethelm-NFord}.

\begin{lemma}\label{ch4-lm1}
  Assume the hypotheses of Theorem 4.1. 
	The function $y \in C[0, h]$ is a solution of the IVP (\ref{eq:pb-de1}) and (\ref{eq:pb-ic}), if and only if, it is a solution of the nonlinear Volterra integral equation of the second kind
\begin{align}
y(x) = \sum_{k=0}^{m-1} \frac{x^k}{k!} y^{(k)}(0) 
				+ \frac{\rho^{1-\alpha}}{\Gamma(\alpha)} \int_0^x (x^\rho - \tau^\rho)^{\alpha -1} \, \tau^{\rho - 1}\,f(\tau, y(\tau)) \, d\tau\ \label{eq:ch4-eq2}
\end{align}
where $m = \lceil \alpha \rceil$ and $\alpha \in \mathbb{R}$.
\end{lemma}
\begin{proof}
First assume that $y$ is a solution of the integral equation. Then we have
\begin{equation*}
y(x) = \sum_{k=0}^{m-1} \frac{x^k}{k!} y^{(k)}(0) + {}^\rho_c I^\alpha_0 f(\cdot, y(\cdot))(x).
\end{equation*}
After applying the differential operator ${}^\rho_c D^\alpha_0$ to both sides of this relation and by the Theorem 4.2 of \cite{key-00}, we obtain that $y$ also solves the differential equation (\ref{eq:pb-de1}). Now to show that the Volterra integral equation satisfies the initial conditions, we let $0 \le j \le m - 1$, and apply the operator $D^j$ to the integral equation, to obtain
\begin{equation*}
D^j y(x) = \sum_{k=0}^{m-1} \frac{D^j x^k}{k!} y^{(k)}(0) + D^j {}^\rho_c I^j_0 \,{}^\rho_c I^{\alpha -j}_0 f(\cdot, y(\cdot))(x).
\end{equation*}
by virtue of the semigroup property of the generalized fractional integration \cite{key-0}. It is clear that summands vanish identically for $j > k$ and for $j < k$ when $x = 0$. Since the integral is zero when $x = 0$, we have $D^j y(0) = y^{(j)}_0$ as required by the initial condition. Hence $y$ solves the given IVP. 

On the other hand, if $y$ is a continuous solution of the IVP, define $z(x) := f(x, y(x))$. Then $z \in C[0, h],$ by the assumptions on $y$ and $f$. Then, using the definition of the Caputo-type differential operator, the differential equation (\ref{eq:pb-de1}) can be written as 
\begin{align*}
 z(x) &= f(x,y(x))= {}^\rho D^\alpha_{0+}\bigg(y -\sum_{k=0}^{m-1} \frac{x^k}{k!}\,y_0^{(k)}\bigg)(x)\\
&= \big(x^{1-\rho} \,D \big)^m\,\, \big({}^\rho \mathcal{I}^{m-\alpha}_{0+}\big)\bigg(y -\sum_{k=0}^{m-1} \frac{x^k}{k!}\,y_0^{(k)}\bigg)(x).
\end{align*}
Since we have that ${}^\rho \mathcal{I}^m_{0+} \big(x^{1-\rho}\,D \big)^m g(x) = g(x)$ for $m \in \mathbb{N}, \,\, g \in C[0, h]$ and we are dealing with continuous functions, we may apply the operator ${}^\rho \mathcal{I}^m_{0+}$ to both sides of the equation and obtain
\[
{}^\rho \mathcal{I}^m_{0+}z(x)= {}^\rho \mathcal{I}^{m-\alpha}_{0+}\bigg(y -\sum_{k=0}^{m-1} \frac{x^k}{k!}\,y_0^{(k)}\bigg)(x).
\]
Applying the operator ${}^\rho D^{m-\alpha}_{0+}$ to this equation we find
\begin{align*}
 y(x) -\sum_{k=0}^{m-1} \frac{x^k}{k!}\,y_0^{(k)}(x) &= {}^\rho D^{m-\alpha}_{0+}\, {}^\rho \mathcal{I}^m_{0+}z(x) =  \big(x^{1-\rho}\,D \big)^1 \,{}^\rho \mathcal{I}^{1-m+\alpha}_{0+} \, {}^\rho \mathcal{I}^m_{0+}z(x)\\
&={}^\rho \mathcal{I}^\alpha_{0+}z(x).
\end{align*}
This completes the proof of the Lemma 4.1. 
\end{proof}

Now, we give the proof of the Theorem 4.1, 
which is a generalization of the Theorem 6.1 of \cite{AFDE}.

\begin{proof}(Theorem 4.1) 
If $M = 0$ then $f(x, y) = 0$ for all $(x, y) \in G$. In this case it is clear by direct substitution that the function $y:[0, h] \rightarrow \mathbb{R}$ with $y(x) = \sum_{k=0}^{m-1} y_0^{(k)} x^k/k!$ is a solution of the IVP (\ref{eq:pb-de1}) and (\ref{eq:pb-ic}). Hence a solution exists in this case.

For $M \ne 0$, we apply Lemma 4.1 
and prove that IVP (\ref{eq:pb-de1}) and (\ref{eq:pb-ic}) is equivalent to the Volterra integral equation (8). 
Define the polynomial $T$ that satisfies the initial condition, namely,
\begin{equation}
   T(x) := \sum_{k=0}^{m-1} \frac{x^k}{k!}\,y_0^{(k)} 
\label{ch4-poly1}
\end{equation}
and the set $U := \{y \in C[0, h]: \|y - T\|_\infty \leq K\}$. It is clear that $U$ is a closed and convex subset of the Banach space of all continuous functions on $[0,h]$, equipped with the Chebyshev norm. Hence, $U$ is a Banach space. $U$ is nonempty, since $T \in U$. We define the operator $A$ on this set $U$ by
\begin{equation}
  \big(Ay\big)(x) := T(x) + \frac{\rho^{1-\alpha}}{\Gamma(\alpha)} \int_0^x (x^\rho - t^\rho)^{\alpha -1} \, t^{\rho - 1}\,f(t, y(t)) \, dt.
\label{eq:A}
\end{equation}
Then, the Volterra equation 
(8) can be written as $y = Ay$ and thus, we have to show that $A$ has a fixed point. This is done by the Schauder's Second Fixed Point Theorem. We first show that $U$ is closed, that is, $Ay \in U$ for $y \in U$. We begin by noting that, for $0 \le x_1 \le x_2 \le h$,
\begin{align}
 \big|(Ay)(x_1) - (Ay)(x_2) + T(x_2) - T(x_1)\big| 
&=\frac{\rho^{1-\alpha}}{\Gamma(\alpha)} \bigg|\int_0^{x_1} (x_1^\rho - t^\rho)^{\alpha -1} \, t^{\rho - 1}\,f(t, y(t)) dt \nonumber\\
&\hspace{3cm} - \int_0^{x_2} (x_2^\rho - t^\rho)^{\alpha -1} \, t^{\rho - 1}\,f(t, y(t)) dt\bigg|\nonumber\\
&=\frac{\rho^{1-\alpha}}{\Gamma(\alpha)} \bigg|\int_0^{x_1}\big[(x_1^\rho - t^\rho)^{\alpha -1} - (x_2^\rho - t^\rho)^{\alpha -1}\big]\, t^{\rho - 1}\,f(t, y(t)) dt \nonumber\\
&\hspace{3cm} + \int_{x_1}^{x_2} (x_2^\rho - t^\rho)^{\alpha -1} \, t^{\rho - 1}\,f(t, y(t)) dt\bigg|\nonumber\\
&\le \frac{M \rho^{1-\alpha}}{\Gamma(\alpha)}\bigg(\int_0^{x_1}\big|(x_1^\rho - t^\rho)^{\alpha -1} - (x_2^\rho - t^\rho)^{\alpha -1}\big|\, t^{\rho - 1} dt \nonumber\\
&\hspace{3cm}+ \int_{x_1}^{x_2} (x_2^\rho - t^\rho)^{\alpha -1} \, t^{\rho - 1}dt\bigg)_.\nonumber
\end{align}
 The second integral in the right-hand side of the last inequality has the value $\tfrac{1}{\rho\alpha}(x_2^\rho - t^\rho)^\alpha$. For the first integral, consider the three cases $\alpha < 1, \,\alpha = 1$ and $\alpha >1 $, separately. In the case $\alpha = 1$, the integral has the value zero. For $\alpha < 1$, we have $(x_1^\rho - t^\rho)^{\alpha-1} \ge (x_2^\rho - t^\rho)^{\alpha-1}$. Thus,
\begin{align*}
  \int_0^{x_1}\big|(x_1^\rho - t^\rho)^{\alpha -1} - (x_2^\rho - t^\rho)^{\alpha -1}\big|\, t^{\rho - 1} dt 
&= \int_0^{x_1}\big[(x_1^\rho - t^\rho)^{\alpha -1} - (x_2^\rho - t^\rho)^{\alpha -1}\big]\, t^{\rho - 1} dt\\
&=\frac{1}{\rho\alpha}(x_1^\rho\alpha - x_2^\rho\alpha) + \frac{1}{\rho\alpha}(x_2^\rho - x_1^\rho)^\alpha \\
&\le \frac{1}{\rho\alpha}(x_2^\rho - x_1^\rho)^\alpha_.
\end{align*} 
Finally, if $\alpha > 1$ then $(x_1^\rho - t^\rho)^{\alpha-1} \le (x_2^\rho - t^\rho)^{\alpha-1}$ and hence
\begin{align*}
  \int_0^{x_1}\big|(x_1^\rho - t^\rho)^{\alpha -1} - (x_2^\rho - t^\rho)^{\alpha -1}\big|\, t^{\rho - 1} dt 
&= \int_0^{x_1}\big[(x_2^\rho - t^\rho)^{\alpha -1} - (x_1^\rho - t^\rho)^{\alpha -1}\big]\, t^{\rho - 1} dt\\
&=\frac{1}{\rho\alpha}(x_2^\rho\alpha - x_1^\rho\alpha) - \frac{1}{\rho\alpha}(x_2^\rho - x_1^\rho)^\alpha \\
&\le \frac{1}{\rho\alpha}(x_2^\rho\alpha - x_1^\rho\alpha)_.
\end{align*} 
Combining these results, we have  
\begin{align}
  \big|(Ay)(x_1) - (Ay)(x_2) &+ T(x_2) - T(x_1)\big| 
\le \begin{cases}
      \frac{2M}{\rho^\alpha\Gamma(\alpha +1)}(x_2^\rho - x_1^\rho)^\alpha \label{ch4-eq3}\\
  \frac{M}{\rho^\alpha\Gamma(\alpha +1)} \big[(x_2^\rho - x_1^\rho)^\alpha + x_2^{\rho\alpha} - x_1^{\rho\alpha},
\end{cases}
\end{align}
if $\alpha \le 1$ and $\alpha > 1$, respectively. In either case, the expression on the right-hand side of (\ref{ch4-eq3}) converges to $0$ as $x_2 \rightarrow x_1$, which proves that $Ay$ is a continuous function, since the polynomial $T(x)$ itself is continuous. It is also true that for $y \in U$ and $x \in [0, h]$, 
\begin{align}
  \big|(Ay)(x) - T(x)\big| &= \frac{\rho^{1-\alpha}}{\Gamma(\alpha)} \bigg|\int_0^x (x^\rho - t^\rho)^{\alpha -1} \, t^{\rho - 1}\,f(t, y(t)) \, dt\bigg| \nonumber\\
  &\le \frac{M}{\rho^\alpha\Gamma(\alpha +1)} x^{\rho\alpha} \le \frac{M}{\rho^\alpha\Gamma(\alpha +1)} h^{\rho\alpha} \nonumber\\
  &\le \frac{M}{\rho^\alpha\Gamma(\alpha +1)} \cdot \frac{\rho^\alpha K \Gamma(\alpha +1)}{M} = K.
\end{align}
by the definition of $h$. Thus, we have $Ay \in U$ if $y \in U$, i.e. $A$ maps the set $U$ into itself. We only left to show that $A(U) := \{Au : u \in U\}$ is relatively compact. This is done by the use of Arzel\`{a}-Ascoli Theorem. To show that the set $A(U)$ is uniformly bounded, let $z \in A(U)$. We see that, for all $x \in [0, h]$,
\begin{align*}
 |z(x)| &= \big|(Ay)(x)\big| \nonumber\\
 &\le \|T\|_\infty + \frac{\rho^{1-\alpha}}{\Gamma(\alpha)} \int_0^x (x^\rho - t^\rho)^{\alpha -1} \, t^{\rho - 1}|f(t, y(t))| \, dt\\
&\le \|T\|_\infty + \frac{1}{\rho^\alpha\Gamma(\alpha +1)} Mh^\alpha \le \|T\|_\infty + K,
\end{align*}
which is the required boundedness property. The equicontinuity property can be derived easily from (\ref{ch4-eq3}) above. For $0 \le x_1 \le x_2 \le h$, we proved in the case $\alpha \le 1$ that
\[
 \big|(Ay)(x_1 - x_2) + T(x_2) - T(x_1)\big| \le \frac{2M}{\rho^\alpha\Gamma(\alpha +1)}(x_2^\rho - x_1^\rho)^\alpha_.   
\]
After using the Triangle Inequality and the Mean Value Theorem, we have that
\begin{align*}
\big|(Ay)(x_1) - (Ay)(x_2)\big| 
& \le | T(x_1) - T(x_2)| + \frac{2M}{\rho^\alpha\Gamma(\alpha +1)}(x_2^\rho - x_1^\rho)^\alpha \\
& = | T(x_1) - T(x_2)| + \frac{2M}{\Gamma(\alpha +1)} (x_2 - x_1)^\alpha \zeta^{\alpha(\rho - 1)}
\end{align*}
for some $\zeta \in [x_1, x_2] \subseteq [0, h].$ Thus, if $|x_2 - x_1| < \delta$, we have
\[
\big|(Ay)(x_1) - (Ay)(x_2)\big| \le M^{'}\delta + \frac{2M}{\Gamma(\alpha +1)} \delta^\alpha h^{\alpha(\rho - 1)}
\]
for some $M^{'} > 0$, since $T(x)$ is uniformly continuous in the closed interval $[0, h]$. Noting that the expression on the right-hand side is independent of $y, x_1$ and $x_2$, we see that the set $A(U)$ equicontinous. The case $\alpha > 1$ can be treated similarly. In either case the Arzel\`{a}-Ascoli Theorem yields that $A(U)$ is relatively compact, and hence Schauder's second Fixed Point Theorem asserts that $A$ has a fixed point. This fixed point is the required solution of the IVP (\ref{eq:pb-de1}) and (\ref{eq:pb-ic}). This completes the proof.
\end{proof}

\section{Uniqueness results for Fractional Order ODEs}
Now we discuss the uniqueness results for single term fractional order differential equations. Before we state the main theorem of this section we notice the following property of the operator $A$(defined in \ref{eq:A}). Thus, let $y_1, y_2 \in C[0, h]\subseteq [0, x]$ and suppose there exists a constant $L > 0$ independent of $x, y_1$ and $y_2$ such that $|f(x, y_1) - f(x, y_2)| \le L|y_1 - y_2|$ for all $x \in [0, h]$. Then we have
\begin{align*}
\|Ay_1 -Ay_2\|_{L_\infty[0, x]} 
&= \frac{\rho^{1-\alpha}}{\Gamma(\alpha)} \sup_{0 \le \omega \le x} \bigg|\int_0^\omega (\omega^\rho - t^\rho)^{\alpha -1} \, t^{\rho - 1}\,\big[f(t, y_1(t)) - f(t, y_2(t))\big] \, dt \bigg| \\
&\le \frac{L\rho^{1-\alpha}}{\Gamma(\alpha)} \,\sup_{0 \le \omega \le x} \bigg|\int_0^\omega (\omega^\rho - t^\rho)^{\alpha -1} \, t^{\rho - 1}\,\big|y_1(t) - y_2(t)\big| \, dt \bigg| \\
&\le \frac{L\rho^{1-\alpha}}{\Gamma(\alpha)} \|y_1 - y_2\|_{L_\infty[0, x]} \, \sup_{0 \le \omega \le x} \bigg|\int_0^\omega (\omega^\rho - t^\rho)^{\alpha -1} \, t^{\rho - 1} \, dt \bigg| \\
&\le \frac{L\rho^{1-\alpha}}{\Gamma(\alpha)} \|y_1 - y_2\|_{L_\infty[0, x]} \, \sup_{0 \le \omega \le x} \bigg|\frac{1}{\rho\alpha}(\omega^\rho - t^\rho)^{\alpha} \Big|^\omega_0 \bigg| \\
&= \frac{L\big(\frac{x^\rho}{\rho}\big)^{\alpha}}{\Gamma(1+\alpha)} \|y_1 - y_2\|_{L_\infty[0, x]}.
\end{align*}
Next, we have the following result;
\begin{theorem}\label{ch4-induc}
Let $A$ and $U$ be defined as in Theorem 4.1 
Also let $j \in \mathbb{N}_0, \,\, x \in [0, h]$ and $y,\tilde{y} \in U$. Suppose $f$ satisfies the Lipschitz condition with respect to the second variable with the Lipschitz constant $L$. Then
\begin{equation}\label{eq:induc}
\|A^jy -A^j\tilde{y} \|_{L_\infty[0, x]} \le \frac{L^j\big(\frac{x^\rho}{\rho}\big)^{\alpha j}}{\Gamma(1+\alpha j)} \|y - \tilde{y} \|_{L_\infty[0, x]}.
\end{equation}
\end{theorem}
\begin{proof}
The proof is by induction. In the case $j=0$, the statement is trivial. Assume (\ref{eq:induc}) is true for $j-1 \in \mathbb{N}_0$. We begin by writing
\begin{align*}
\|A^jy -A^j\tilde{y} \|_{L_\infty[0, x]} &= \|A(A^{j-1}y) -A(A^{j-1}\tilde{y}) \|_{L_\infty[0, x]} \\
&= \frac{\rho^{1-\alpha}}{\Gamma(\alpha)} \sup_{0 \le \omega \le x} \bigg|\int_0^\omega (\omega^\rho - t^\rho)^{\alpha -1} \, t^{\rho - 1}\nonumber\\
&\hspace{2cm} \cdot\big[f(t, A^{j-1}y(t)) - f(t, A^{j-1}\tilde{y}(t))\big] \, dt \bigg|.
\end{align*}
Now, using the Lipschitz assumption of $f$ and the induction hypothesis:
\begin{align*}
\|A^jy -A^j\tilde{y}\|_{L_\infty[0, x]} 
&\le \frac{L\rho^{1-\alpha}}{\Gamma(\alpha)} \,\sup_{0 \le \omega \le x} \bigg|\int_0^\omega (\omega^\rho - t^\rho)^{\alpha -1} \, t^{\rho - 1}\,\big|A^{j-1}y(t) - A^{j-1}\tilde{y}(t)\big| \, dt \bigg| \\
&\le \frac{L\rho^{1-\alpha}}{\Gamma(\alpha)} \int_0^x (x^\rho - t^\rho)^{\alpha -1} \, t^{\rho - 1} \\
&\hspace{3.5cm}\cdot \sup_{0 \le \omega \le t} \big|A^{j-1}y(\omega) - A^{j-1}\tilde{y}(\omega)\big| \, dt \\
&\le \frac{L^j\rho^{1-\alpha j}}{\Gamma(\alpha)\Gamma(1+\alpha(j-1))} \int_0^x (x^\rho - t^\rho)^{\alpha -1} \, t^{\rho +\rho\alpha(j-1)- 1} \\
&\hspace{3.5cm} \cdot \sup_{0 \le \omega \le t} \big|y(\omega) - \tilde{y}(\omega)\big| \, dt \\
&\le \frac{L^j\rho^{1-\alpha j}}{\Gamma(\alpha)\Gamma(1+\alpha(j-1))} \, \sup_{0 \le \omega \le x} \big|y(\omega) - \tilde{y}(\omega)\big|\\
&\hspace{3.5cm}\cdot \int_0^x (x^\rho - t^\rho)^{\alpha -1} \, t^{\rho +\rho\alpha(j-1)- 1} \, dt \\
&\le \frac{L^j\rho^{1-\alpha j}}{\Gamma(\alpha)\Gamma(1+\alpha(j-1))} \,\|y - \tilde{y}\|_{L_\infty[0, x]}\\
&\hspace{3.5cm}\cdot \frac{\Gamma(\alpha)\Gamma(1+\alpha(j-1))}{\Gamma(1+\alpha j)} \frac{x^{\rho\alpha j}}{\rho} \\
&= \frac{L^j\big(\frac{x^\rho}{\rho}\big)^{\alpha}}{\Gamma(1+\alpha j)} \|y - \tilde{y}\|_{L_\infty[0, x]} \, .
\end{align*} 
This concludes the proof of Theorem 5.1. 
\end{proof}

Next we give a uniqueness theorem that corresponds to the well-known Picard-Lindel\"{o}f theorem for the uniqueness of solution of ordinary differential equations. 
\begin{theorem}[Uniqueness] 
Let $y_0^{(0)}, \ldots, y_0^{(m-1} \in \mathbb{R}, \,\, K > 0 \,\, \text{and} \,\,h^* > 0.$ Also let $\alpha > 0$ and $m = \lceil \alpha \rceil$. Define the set $G$ as in Theorem 4.1 
and let the function $f:G \rightarrow \mathbb{R}$ be continuous and satisfies a Lipschitz condition with respect to the second variable, i.e. 
\begin{equation}
|f(x, y_1) - f(x, y_2)| \le L|y_1 - y_2|
\label{ch4-Lip}
\end{equation}
for some constant $L > 0$ independent of $x, y_1, \text{and}\; y_2$. Then, denoting $h$ as in Theorem  4.1, 
there exists a unique solution $y \in C[0, h]$ for the IVP (\ref{eq:pb-de1}) and (\ref{eq:pb-ic}).
\end{theorem}
\begin{proof}
For the following proof we use a similar argument as applied to the proof of the Theorem 6.5 of \cite{AFDE}. According to Theorem 4.1, 
the IVP (\ref{eq:pb-de1}) and (\ref{eq:pb-ic}) has a solution. In order to prove the uniqueness, we adopt Theorem 5.1. 
In particular, we use polynomial $T$ as defined in (\ref{ch4-poly1}) and the operator $A$ as defined in (\ref{eq:A}) and recall that it maps the nonempty, convex and closed set $U = \{y \in C[0, h]: \|y - T\|_\infty \le K\}$ to itself. We apply Weissinger's Fixed Point Theorem 
to prove that $A$ has a unique fixed point. Let $j \in \mathbb{N}_0, x \in [0, h]$ and $y_1, y_2 \in U$. Then, using (\ref{eq:induc}) and taking the Chebyshev norms on the interval $[0, h]$, we have
\begin{equation*}
\|A^jy -A^j\tilde{y} \|_\infty \le \frac{L^j\big(\frac{h^\rho}{\rho}\big)^{\alpha j}}{\Gamma(1+\alpha j)} \|y - \tilde{y} \|_\infty.
\end{equation*}
Let $\omega_j = L^j(h^\rho/\rho)^{\alpha j}/\Gamma(1+\alpha j)$. In order to apply the theorem, we only need to show that the series $\sum_{j=0}^\infty \omega_j $ converges. It is clear from noticing that $\omega_j$ is simply the power series representation of the Mittag-Leffler function $E_\alpha(L(h^{\rho}/\rho)^\alpha)$, hence the series converges. This completes the proof. 
\end{proof}

\begin{remark}
The above proof gives a constructive method to find a sequence of Picard iterations that converges to the exact solution of the IVP. 
\end{remark}




\section*{Acknowledgements}
The author would like to thank Prof. Jerzy Kocik, the Department of Mathematics of Southern Illinois University, Carbondale, USA for his valuable comments and suggestions during the writing of his dissertation. 


\bibliographystyle{abbrv}

\end{document}